\documentclass[aop,preprint]{imsart}

\RequirePackage[OT1]{fontenc}
\RequirePackage{amsthm,amsmath,natbib}
\RequirePackage[colorlinks]{hyperref}
\RequirePackage{hypernat}

\RequirePackage[OT1]{fontenc}
\RequirePackage{amsthm,amsmath,natbib}
\RequirePackage[colorlinks]{hyperref}
\RequirePackage{hypernat}

\numberwithin{equation}{section}

\usepackage{amssymb}
\usepackage{comment}
\usepackage{appendix}
\usepackage{graphicx}
\usepackage{subfigure}
\usepackage{enumerate}
\usepackage{comment}

\input xy
\xyoption{all}

\startlocaldefs

\newtheorem{theorem}{Theorem}[section]
\newtheorem{lemma}[theorem]{Lemma}
\newtheorem{proposition}[theorem]{Proposition}

\newtheorem{assumption}[theorem]{Assumption}
\newtheorem{remark}[theorem]{Remark}

\newcommand{\E}{\mathbf{E}}

\newcommand{\N}{\mathbf{N}}
\newcommand{\Z}{\mathbf{Z}}
\newcommand{\p}{\mathbf{P}}

\newcommand{\CC}{\mathcal {C}}

\newcommand{\CF}{\mathcal {F}}
\newcommand{\CG}{\mathcal {G}}

\newcommand{\CL}{\mathcal {L}}

\newcommand{\CU}{\mathcal {U}}

\newcommand{\CX}{\mathcal {X}}

\newcommand{\one}{\mathbf{1}}

\newcommand{\ol}{\overline}

\newcommand{\hit}{{\rm hit}}

\newcommand*{\wt}{\widetilde}

\newcommand*{\be}{\begin{equation}}
\newcommand*{\ee}{\end{equation}}
\newcommand*{\ba}{\begin{aligned}}
\newcommand*{\ea}{\end{aligned}}
\newcommand*{\barr}{\begin{array}{c}}
\newcommand*{\earr}{\end{array}}

\newcommand*{\ind}{\mathbf{1}}
\newcommand*{\trel}{t_{{\rm rel}}}
\newcommand*{\tu}{t_{{\rm u}}}
\newcommand*{\tcov}{t_{{\rm cov}}}
\newcommand*{\thit}{t_{{\rm hit}}}
\newcommand*{\tmix}{t_{{\rm mix}}}
\newcommand*{\TV}{{\rm TV}}

\endlocaldefs

\begin{document}
\begin{frontmatter}
\title{Uniform mixing time for Random Walk on Lamplighter Graphs}
\runtitle{Uniform mixing time for lamplighter walks}

\begin{aug}

\runauthor{J\'ulia Komj\'athy, Jason Miller, and Yuval Peres}

\author{J\'ulia Komj\'athy \thanks{J. Komj\'athy was supported by the grant
 KTIA-OTKA  $\#$ CNK 77778, funded by the Hungarian National Development Agency (NF\"U)  from a
source provided by KTIA and also supported by the grant T\'AMOP -
$4.2.2.B-10/1--2010-0009$."}}
\address{Technical University of Budapest,
\tt{komyju@math.bme.hu}}

\author{Jason Miller}
\address{Microsoft Research
\tt{jmiller@microsoft.com}}

\author{Yuval Peres}
\address{Microsoft Research
\tt{peres@microsoft.com}}

\affiliation{Technical University of Budapest and Microsoft Research}

\end{aug}
\date{\today}

\begin{abstract}
Suppose that $\CG$ is a finite, connected graph and $X$ is a lazy random walk on $\CG$.  The lamplighter chain $X^\diamond$ associated with $X$ is the random walk on the wreath product $\CG^\diamond = \Z_2 \wr \CG$, the graph whose vertices consist of pairs $(f,x)$ where $f$ is a labeling of the vertices of $\CG$ by elements of $\Z_2$ and $x$ is a vertex in $\CG$.  There is an edge between $(f,x)$ and $(g,y)$ in $\CG^\diamond$ if and only if $x$ is adjacent to $y$ in $\CG$ and $f(z) = g(z)$ for all $z \neq x,y$.  In each step, $X^\diamond$ moves from a configuration $(f,x)$ by updating $x$ to $y$ using the transition rule of $X$ and then sampling both $f(x)$ and $f(y)$ according to the uniform distribution on $\Z_2$; $f(z)$ for $z \neq x,y$ remains unchanged.  We give matching upper and lower bounds on the uniform mixing time of $X^\diamond$ provided $\CG$ satisfies mild hypotheses.  In particular, when $\CG$ is the hypercube $\Z_2^d$, we show that the uniform mixing time of $X^\diamond$ is $\Theta(d 2^d)$.  More generally, we show that when $\CG$ is a torus $\Z_n^d$ for $d \geq 3$, the uniform mixing time of $X^\diamond$ is $\Theta(d n^d)$ uniformly in $n$ and $d$.  A critical ingredient for our proof is a concentration estimate for the local time of random walk in a subset of vertices.
\end{abstract}



\end{frontmatter}

\maketitle

\section{Introduction}

Suppose that $\CG$ is a finite graph with vertices $V(\CG)$ and edges $E(\CG)$, respectively.  Let $\CX(\CG) = \{f \colon V(\CG) \to \Z_2\}$ be the set of markings of $V(\CG)$ by elements of $\Z_2$.  The wreath product $\CG^\diamond = \Z_2 \wr \CG$ is the graph whose vertices are pairs $(f,x)$ where $f \in \CX(\CG)$ and $x \in V(\CG)$.  There is an edge between $(f,x)$ and $(g,y)$ if and only if $\{x,y\} \in E(\CG)$ and $f(z) = g(z)$ for all $z \notin \{x,y\}$.  Suppose that $P$ is a transition matrix for a Markov chain on $\CG$.  The lamplighter walk $X^\diamond$ (with respect to the transition matrix $P$) is the Markov chain on $\CG^\diamond$ which moves from a configuration $(f,x)$ by
\begin{enumerate}
\item picking $y$ adjacent to $x$ in $\CG$ according to $P$, then
\item updating each of the values of $f(x)$ and $f(y)$ independently according to the uniform measure on $\Z_2$.
\end{enumerate}
The lamp states at all other vertices in $\CG$ remain fixed.  It is easy to see that if $P$ is ergodic and reversible with stationary distribution $\pi_P$ then the unique stationary distribution of $X^\diamond$ is the product measure
\[ \pi\big((f,x)\big)= \pi_P(x) 2^{-|\CG|},\]
and $X^\diamond$ is itself reversible.  In this article, we will be concerned with the special case that $P$ is the transition matrix for the \emph{lazy random walk} on $\CG$ in order to avoid issues of periodicity.  That is, $P$ is given by
\begin{equation}
\label{eqn::lazy_rw_definition}
P(x,y) = \begin{cases} \frac{1}{2} \text{ if } x = y,\\ \frac{1}{2d(x)} \text{ if } \{x,y\} \in E(\CG), \end{cases}
\end{equation}
for $x,y \in V(\CG)$ and where $d(x)$ is the degree of $x$.

\begin{figure}
\includegraphics[width=0.35\textwidth]{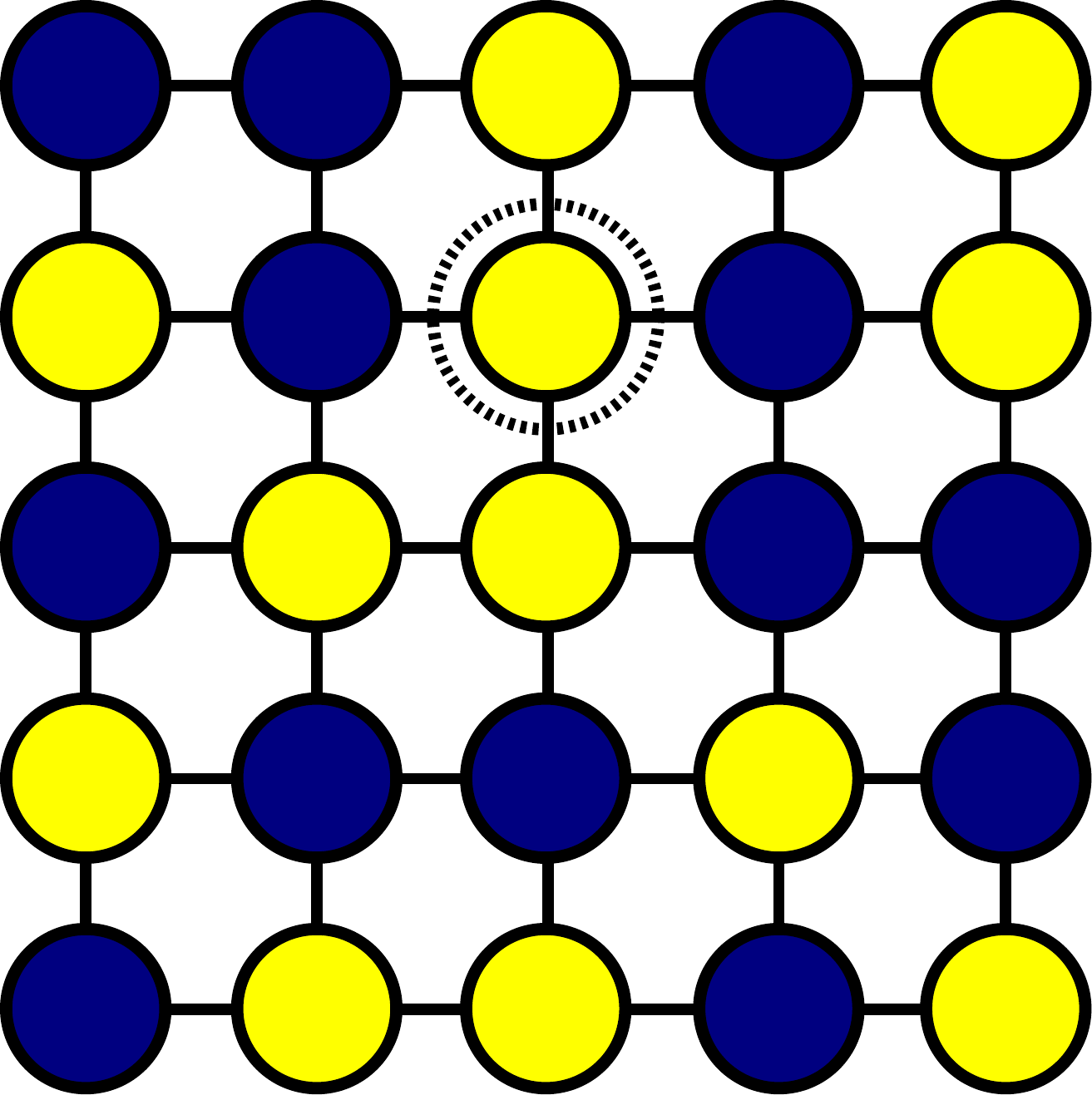}
\caption{A typical configuration of the lamplighter over a $5 \times 5$ planar grid.  The colors indicate the state of the lamps and the dashed circle gives the position of the lamplighter.}
\end{figure}

\subsection{Main Results}

Let $P$ be the transition kernel for lazy random walk on a finite, connected graph $\CG$ with stationary distribution $\pi$.  The $\epsilon$-\emph{uniform mixing time} of $\CG$ is given by
\begin{equation}
\label{eqn::tu_definition}
 \tu(\epsilon,\CG) = \min\left\{ t \geq 0: \max_{x,y \in V(\CG)} \left|\frac{P^t(x,y)-\pi(y)}{\pi(y)}\right| \leq \epsilon \right\}.
\end{equation}
Throughout, we let $\tu(\CG) = \tu((2e)^{-1},\CG)$.  The main result of this article is a general theorem which gives matching upper and lower bounds of $\tu(\CG^\diamond)$ provided $\CG$ satisfies several mild hypotheses.  One important special case of this result is the hypercube $\Z_2^d$ and, more generally, tori $\Z_n^d$ for $d \geq 3$.  These examples are sufficiently important that we state them as our first theorem.

\begin{theorem}
\label{thm::cycles}
There exists constants $C_1,C_2 > 0$ such that
\[ C_1 \leq \frac{\tu((\Z_2^d)^\diamond)}{d 2^d} \leq C_2 \text{ for all } d.\]
More generally,
\[ C_1 \leq \frac{\tu((\Z_n^d)^\diamond)}{d n^{d+2}} \leq C_2 \text{ for all } n \geq 2 \text{ and } d \geq 3.\]
\end{theorem}
\noindent Prior to this work, the best known bound \cite{PR04} for $\tu( (\Z_2^d)^\diamond)$ was
\[ C_1 d 2^d \leq \tu( (\Z_2^d)^\diamond) \leq C_2  \log(d) d 2^d\]
for $C_1,C_2 > 0$.

In order to state our general result, we first need to review some basic terminology from the theory of Markov chains.  The \emph{relaxation time} of $P$ is
\begin{equation}
\label{eqn::trel_definition}
 \trel(\CG) = \frac{1}{1-\lambda_0}
\end{equation}
where $\lambda_0$ is the second largest eigenvalue of $P$.  The \emph{maximal hitting time} of $P$ is
\begin{equation}
\label{eqn::thit_definition}
 \thit(\CG) = \max_{x,y\in V(\CG)} \E_x[\tau_y],
\end{equation}
where $\tau_y$ denotes the first time $t$ that $X(t) = y$ and $\E_x$ stands for the expectation under the law in which $X(0) = x$.  The Green's function $G(x,y)$ for $P$ is
\begin{equation}
\label{eqn::green_definition}
 G(x,y) = \E_x \left[ \sum_{t=0}^{\tu(\CG)} \one_{\{X(t) = y\}} \right] = \sum_{t=0}^{\tu(\CG)} P^t(x,y),
\end{equation}
i.e. the expected amount of time $X$ spends at $y$ up to time $\tu$ given $X(0) = x$.  For each $1 \leq n \leq |\CG|$, we let
\begin{equation}
\label{eqn::g_star_definition}
 G^*(n) = \max_{\stackrel{S \subseteq V(\CG)}{|S| = n}} \max_{z \in S} \sum_{y \in S} G(z,y).
\end{equation}
This is the maximal expected time $X$ spends in a set $S \subseteq V(\CG)$ of size $n$ before the uniform mixing time.  This quantity is related to the hitting time of subsets of $V(\CG)$.  Finally, recall that $\CG$ is said to be vertex transitive if for every $x,y \in V(\CG)$ there exists an automorphism $\varphi$ of $\CG$ with $\varphi(x) = y$.  Our main result requires the following hypothesis.

\begin{assumption}
\label{assump::main}
$\CG$ is a finite, connected, vertex transitive graph and $X$ is a lazy random walk on $\CG$.  There exists constants $K_1,K_2,K_3 > 0$ such that
\begin{enumerate}[(A)]
\item \label{assump::main::hitting} $\thit(\CG) \leq K_1|\CG|$,
\item \label{assump::main::green_exponent} $2K_2 (5/2)^{K_2} (G(x,y))^{K_2} \leq \exp\left(- \frac{\tu(\CG)}{\trel(\CG)} \right)$,
\item \label{assump::main::green_ratio} $G^*(n^*) \leq K_3(\trel(\CG) + \log|\CG|)/(\log n^*)$
\end{enumerate}
where $n^* = 4 K_2 \tu(\CG) / G(x,y)$ for $x,y \in V(\CG)$ adjacent.
\end{assumption}

The general theorem is:

\begin{theorem}
\label{thm::main}
Let $\CG$ be any graph satisfying Assumption \ref{assump::main}.  There exists constants $C_1,C_2$ depending only on $K_1,K_2,K_3$ such that
\begin{align}
 C_1  \leq \frac{\tu(\CG^\diamond)}{|\CG| \big( \trel(\CG) +\log|\CG|\big)}  \leq C_2 \label{eqn::tu_bound}
\end{align}
\end{theorem}

The lower bound is proved in \cite[Theorem 1.4]{PR04}.  The proof of the upper bound is based on the observation from \cite{PR04} that the uniform distance to stationarity can be related to $\E[ 2^{|\CU(t)|}]$ where $\CU(t)$ is the set of vertices in $\CG$ which have not been visited by $X$ by time $t$.  Indeed, suppose that $f$ is any initial configuration of lamps, let $f(t)$ be the state of the lamps at time $t$, and let $g$ be an arbitrary lamp configuration.  Let $W$ be the set of vertices where $f \neq g$.  Let $\CC(t) = V(\CG) \setminus \CU(t)$ be the set of vertices which have been visited by $X$ by time $t$.  With $\p_{(f,x)}$ the probability under which $X^\diamond(0) = (f,x)$, we have that
\[ \p_{(f,x)}[f(t) = g | \CC(t)] =  2^{-|\CC(t)|} \one_{\{ W \subseteq \CC(t)\}}.\]
Since the probability of the configuration $g$ under the uniform measure is $2^{-|\CG|}$, we therefore have
\begin{equation}
\label{eqn::uniform_mix_expression}
 \frac{\p_{(f,x)}[f(t) = g]}{2^{-|\CG|}} = \E_{(f,x)}\big[2^{|\CU(t)|} \one_{\{ W \subseteq \CC(t)\}} \big].
\end{equation}
The right hand side is clearly bounded from above by $\E[2^{|\CU(t)|}]$ (the initial lamp configuration and position of the lamplighter no longer matters).  On the other hand, we can bound \eqref{eqn::uniform_mix_expression} from below by
\[ \p_{(f,x)}[ W \subseteq \CC(t)] \geq \p[ |\CU(t)| = 0] \geq 1-(\E\big[2^{|\CU(t)|} \big] - 1).\]
Consequently, to bound $\tu(\epsilon,\CG^\diamond)$ it suffices to compute
\begin{equation}
\label{eqn::expmom}
 \min\{ t \geq 0 : \E[2^{|\CU(t)|}] \leq 1+\epsilon\}
\end{equation}
since the amount of time it requires for $X$ to subsequently uniformly mix after this time is negligible.

In order to establish \eqref{eqn::expmom}, we will need to perform a rather careful analysis of the process by which $\CU(t)$ is decimated by $X$.  The key idea is to break the process of coverage into two different regimes, depending on the size of $\CU(t)$.  The main ingredient to handle the case when $\CU(t)$ is large is the following concentration estimate of the local time
\[ \CL_S(t) = \sum_{s=0}^t \one_{\{X(s) \in S\}}\]
for $X$ in $S \subseteq V(\CG)$.

\begin{proposition}
\label{prop::local_time}
Let $\lambda_0$ be the second largest eigenvalue of $P$.  Assume $\lambda_0 \geq \tfrac{1}{2}$ and fix $S \subseteq V(\CG)$.  For $C_0 = 1/50$, we have that
\begin{equation}
\label{eqn::loctime_bound} \p_\pi\left[\CL_S(t) \leq  t \frac{\pi (S)}{2}\right] \leq \exp\left( - C_0  t  \frac{\pi(S)}{\trel(\CG)}\right).
\end{equation}
\end{proposition}
Proposition \ref{prop::local_time} is a corollary of \cite[Theorem 1]{LP04}; we consider this sufficiently important that we state it here.  By invoking Green's function estimates, we are then able to show that the local time is not concentrated on a small subset of $S$.  The case when $\CU(t)$ is small is handled via an estimate (Lemma \ref{lem::hit_small_set}) of the hitting time $\tau_S = \min\{t \geq 0 : X(t) \in S\}$ of $S$.

\subsection{Previous Work}

Suppose that $\mu,\nu$ are probability measures on a finite measure space.  Recall that the \emph{total variation distance} between $\mu,\nu$ is given by
\begin{equation}
\label{eqn::tv_definition}
 \| \mu - \nu \|_{\TV} = \max_{A} |\mu(A) - \nu(A)| = \frac{1}{2} \sum_{x} |\mu(x) - \nu(x)|.
\end{equation}
The $\epsilon$-\emph{total variation mixing time} of $P$ is
\begin{equation}
\label{eqn::tmix_definition}
 \tmix(\epsilon,\CG) = \min\left\{ t \geq 0: \max_{x \in V(\CG)} \|P^t(x,\cdot) - \pi\|_{\TV} \leq \epsilon \right\}.
\end{equation}
Let $\tmix(\CG) = \tmix((2e)^{-1},\CG)$.  It was proved \cite[Theorem 1.4]{PR04} by Peres and Revelle that if $\CG$ is a regular graph such that $\thit(\CG) \leq K|\CG|$, there exists constants $C_1,C_2$ depending only on $K$ such that
\[ C_1|\CG|(\trel(\CG) + \log|\CG|) \leq \tu(\CG^\diamond) \leq C_2|\CG|(\tmix(\CG) + \log|\CG|).\]
These bounds fail to match in general.  For example, for the hypercube $\Z_2^d$, $\trel(\Z_2^d) = \Theta(d)$ \cite[Example 12.15]{LPW08} while $\tmix(\Z_2^d) = \Theta(d \log d)$ \cite[Theorem 18.3]{LPW08}.  Theorem \ref{thm::main} says that the lower from \cite[Theorem 1.4]{PR04} is sharp.

Before we proceed to the proof of Theorem \ref{thm::main}, we will mention some other work on mixing times for lamplighter chains.  The mixing time of $\CG^\diamond$ was first studied by H\"aggstr\"om and Jonasson in \cite{HJ97} in the case of the complete graph $K_n$ and the one-dimensional cycle $\Z_n$.  Their work implies a total variation cutoff with threshold $\tfrac{1}{2} \tcov(K_n)$ in the former case and that there is no cutoff in the latter.  Here, $\tcov(\CG)$ for a graph $\CG$ denotes the expected number of steps required by lazy random walk to visit every site in $\CG$.  The connection between $\tmix(\CG^\diamond)$ and $\tcov(\CG)$ is explored further in \cite{PR04}, in addition to developing the relationship between the relaxation time of $\CG^\diamond$ and $\thit(\CG)$, and $\E[2^{|\CU(t)|}]$ and $\tu(\CG^\diamond)$.  The results of \cite{PR04} include a proof of total variation cutoff for $\Z_n^2$ with threshold $\tcov(\Z_n^2)$.  In \cite{MP11}, it is shown that $\tmix((\Z_n^d)^\diamond) \sim \tfrac{1}{2} \tcov(\Z_n^d)$ when $d \geq 3$ and more generally that $\tmix(\CG_n^\diamond) \sim \tfrac{1}{2} \tcov(\CG_n)$ whenever $(\CG_n)$ is a sequence of graphs satisfying some uniform local transience assumptions.

The mixing time of $X^\diamond = (f,X)$ is typically dominated by the first coordinate $f$ since the amount of time it takes for $X$ to mix is negligible compared to that required by $X^\diamond$.  We can sample from $f(t)$ by:
\begin{enumerate}
\item sampling the range $\CC(t)$ of lazy random walk run for time $t$, then
\item marking the vertices of $\CC(t)$ by iid fair coin flips.
\end{enumerate}
Determining the mixing time of $X^\diamond$ is thus typically equivalent to computing the threshold $t$ where the corresponding marking becomes indistinguishable from a uniform marking of $V(\CG)$ by iid fair coin flips.  This in turn can be viewed as a statistical test for the uniformity of the uncovered set $\CU(t)$ of $X$ --- if $\CU(t)$ exhibits any sort of non-trivial systematic geometric structure then $X^\diamond(t)$ is not mixed.  This connects this work to the literature on the geometric structure of the last visited points by random walk \cite{DPRZ_LATE06, DPRZ_COV04, BH91, MP11}.

\subsection{Outline}
The remainder of this article is structured as follows.  In Section \ref{sec::examples}, we will give the proof of Theorem \ref{thm::cycles} by checking the hypotheses of Theorem \ref{thm::main}.  Next, in Section \ref{sec::estimates} we will collect a number of estimates regarding the amount of $X$ spends in and requires to cover sets of vertices in $\CG$ of various sizes.  Finally, in Section \ref{sec::proof_of_main}, we will complete the proof of Theorem \ref{thm::main}.

\section{Proof of Theorem \ref{thm::cycles}}
\label{sec::examples}

We are going to prove Theorem~\ref{thm::cycles} by checking the hypotheses of Theorem~\ref{thm::main}.  We begin by noting that by \cite[Corollary 12.12]{LPW08} and \cite[Section 12.3.1]{LPW08}, we have that
\begin{equation}
\label{eqn::relaxation_asymptotics}
 \trel(\Z_n^d) = \Theta(dn^2).
\end{equation}
By \cite[Example 2, Page 2155]{DSC93}, we know that $\tu(\Z_n) = O(n^2)$.  Hence by \cite[Theorem 2.10]{DSC96}, we have that
\begin{equation}
\label{eqn::uniform_asymptotics}
\tu(\Z_n^d) = O((d \log d)n^2).
\end{equation}

The key to checking parts \eqref{assump::main::hitting}--\eqref{assump::main::green_ratio} of Assumption~\ref{assump::main} are the Green's function estimates which are stated in Proposition~\ref{prop::low_degree_green_bound} (low degree) and Proposition~\ref{prop::high_degree_green_bound} (high degree).  In order to establish these we will need to prove several intermediate technical estimates.  We begin by recording the following facts about the transition kernel $P$ for lazy random walk on a vertex transitive graph $\CG$.  First, we have that
\begin{equation}
\label{eqn::lazy_monotone_distance}
P^t(x,y) \leq P^t(x,x) \text{ for all } x,y.
\end{equation}
To see this, we note that for $t$ even, the Cauchy-Schwarz inequality and the semigroup property imply
\[ P^t(x,y)= \sum_z P^{t/2}(x,z) P^{t/2}(z,y) \leq \sqrt{P^t(x,x) P^t(y,y)} = P^t(x,x).\]
The inequality and final equality use the vertex transitivity of $\CG$ so that $P(x,z) = P(z,x)$ and $P(x,x) = P(y,y)$.  To get the same result for $t$ odd, one just applies the same trick used in the proof of \cite[Proposition 10.18(ii)]{LPW08}.  Moreover, by \cite[Proposition 10.18]{LPW08}, we have that
\begin{equation}
\label{eqn::lazy_monotone_time}
 P^t(x,x) \leq P^s(x,x) \text{ for all } s \leq t.
\end{equation}

The main ingredient in the proof of Proposition~\ref{prop::low_degree_green_bound}, our low degree Green's function estimate, is the following bound for the return probability of a lazy random walk on $\Z^d$.
\begin{lemma}
\label{lem::lazy_transition}
Let $P(x,y;\Z^d)$ denote the transition kernel for lazy random walk on $\Z^d$.  For all $t \geq 1$, we have that
\be\label{eqn::lazy_transition}
P^t(x,x; \Z^d) \le \sqrt{2} \left( \frac{ 4d}{\pi}\right)^{d/2} \frac{1}{t^{d/2}}+ e^{-t/8}.
\ee
\end{lemma}
\begin{proof}
To prove the lemma we first give an upper bound on the transition probabilities for a (non-lazy) simple random walk $Y$ on $\Z^d$.  One can easily give an exact formula for the return probability of $Y$ to the origin of $\Z^d$ in $2t$ steps by counting all of the possible paths from $0$ back to $0$ of length $2t$ (here and hereafter, $P_{\rm NL}(x,y;\Z^d)$ denotes the transition kernel of $Y$):
\[ \ba P^{2t}_{\rm NL}\left( x,x ; \Z^d \right) &=   \sum_{n_1+\dots + n_d=t} \frac{(2t)!}{(n_1!)^2  (n_2!)^2 \cdots (n_d!)^2 }\cdot \frac{1}{(2d)^{2t}}\\
 &=  \frac{1}{(2d)^{2t}} {2t \choose t }\sum_{n_1+\dots + n_d=t} \left(\frac{t!}{n_1! n_2! \cdots n_d! }\right)^2
 \ea \]
 We can bound the sum above as follows, using the multinomial theorem in the second step:
\[ \ba  P^{2t}_{\rm NL}\left( x,x ; \Z^d \right) &\le
 \frac{1}{(2d)^{2t}} {2t \choose t }\left(\!\! \max_{n_1+\dots + n_d=t}\frac{t!}{n_1! \cdots n_d! }\right)\!\!\sum_{n_1+\cdots + n_d=t} \!\!\!\frac{t!}{n_1! \cdots n_d! }\\
&\le \frac{1}{(2d)^{2t}} {2t \choose t } \frac{t!}{\left[(\lfloor t/d \rfloor)!\right]^d}\cdot d^t.
\ea \]
Applying Stirlings formula to each term above, we consequently arrive at
\be
\label{eqn::nonlazy_transition}
P^{2t}_{\rm NL}\left( x,x; \Z^d \right) \le \frac{\sqrt{2}}{(2\pi)^{d/2}}\cdot \frac{d^{d/2}}{t^{d/2}}
\ee

We are now going to deduce from \eqref{eqn::nonlazy_transition} a bound on the return probability for a lazy random walk $X$ on $\Z^d$.  We note that we can couple $X$ and $Y$ so that $X$ is a random time change of $Y$: $X(t) = Y(N_t)$ where $N_t = \sum_{i=0}^t \xi_i$ and the $(\xi_i)$ are iid with $\p[\xi_i = 0] = \p[\xi_i = 1] = \tfrac{1}{2}$ and are independent of $Y$.  Note that $N_t$ is distributed as a binomial random variable with parameters $t$ and $1/2$.  Thus,
\[ \ba P^t(x,x; \Z^d) &= \sum_{i=0}^{t/2} P^{2i}_{\rm NL}(x,x;\Z^d) \p(N_t=2i)\\
&\le \p(N_t< t/4) +  \sqrt{2}\left( \frac{4d}{\pi} \right)^{d/2} \frac{1}{t^{d/2}}, \ea\]
where in the second term we used the monotonicity of the upper bound in \eqref{eqn::nonlazy_transition} in $t$.
The first term can be bounded from above by using the Hoeffding inequality.  This yields the term $e^{-t/8}$ in \eqref{eqn::lazy_transition}.
 \end{proof}

Throughout the rest of this section, we let $|x-y|$ denote the $L^1$ distance between $x,y \in \Z_n^d$.

\begin{proposition}
\label{prop::low_degree_green_bound}
Let $G(x,y)$ denote the Green's function for lazy random walk on $\Z_n^d$.  For each $\delta \in (0,1)$, there exists constants $C_1,C_2,C_3 > 0$ independent of $n,d$ for $d \geq 3$ such that
\[ \ba G(x,y) &\leq \frac{C_1}{d}\left(\frac{ 4d}{\pi}\right)^{d/2}|x-y|^{1-d/2} + C_2 (d \log d) \left(\frac{4d}{\pi}\right)^{d/2} n^{2-d(1-\delta/2)} \\&\ \ +C_3 \big( d^2 \log d  \big) n^2 e^{- n^{\delta}/2} \ea\]
for all $x,y \in \Z_n^d$ distinct.
\end{proposition}
\begin{proof}
Fix $\delta \in (0,1)$.  We first observe that the probability that there is a coordinate in which the random walk wraps around the torus within $t<n^2$ steps can be estimated by using Hoeffding's inequality and a union bound by
\[ d\cdot \p(Z(t) > n ) = d e^{-\frac{n^2}{2t}} \]
where $Z(t)$ is a one dimensional simple random walk on $\Z$.  Let $k=|x-y|$.  Applying \eqref{eqn::lazy_monotone_distance} and \eqref{eqn::lazy_monotone_time} in the second step, and estimating the probability of wrapping around in time $n^{2-\delta}$ in the third term, we see that
\begin{align}
\label{eqn::green_function_split}  G(x,y)
= \sum_{t=k}^{\tu} P^t(x,y)
\leq& \sum_{t=k}^{{n^{2-\delta}}} P^t(x,x; \Z^d) + \tu P^{n^{2-\delta}}(x,x;\Z^d) \\&\ +d \tu  e^{-\frac{n^{\delta}}{2}}. \notag
\end{align}
We can estimate the sum on the right hand side above using Lemma~\ref{lem::lazy_transition}, yielding the first term in the assertion of the lemma.  Applying Lemma~\ref{lem::lazy_transition} again, we see that there exists a constant $C_2$ which does not depend on $n,d$ such that the second term in the right side of \eqref{eqn::green_function_split} is bounded by
\begin{equation}
\label{eqn::low_degree_bound_2}
C_2 (d \log d) \left(\frac{4d}{\pi}\right)^{d/2} n^{2-d(1-\delta/2)}.
\end{equation}
Indeed, the factor $(d \log d) n^2$ comes from \eqref{eqn::uniform_asymptotics} and the other factor comes from Lemma~\ref{lem::lazy_transition}.  Combining proves the lemma.
\end{proof}


Proposition~\ref{prop::low_degree_green_bound} is applicable when $n$ is much larger than $d$.  We now turn to prove Proposition~\ref{prop::high_degree_green_bound}, which gives us an estimate for the Green's function which we will use when $d$ is large.  Before we prove Proposition~\ref{prop::high_degree_green_bound}, we first need to collect the following estimates.

\begin{lemma}
\label{lem::escape_bound}  Suppose that $X$ is a lazy random walk on $\Z_n^d$ for $d \geq 8$ and that $|X(0)|=k \leq \tfrac{d}{8}$.  For each $j \geq 0$, let $\tau_j$ be the first time $t$ that $|X(t)| = j$.  There exists a constant $C_k > 0$ depending only on $k$ such that $\p[\tau_0 < \tau_{2k}] \leq C_k d^{-k}$.  If, instead, $|X(0)|=1$, then there exists a universal constant $p > 0$ such that $\p[ \tau_0 < \tau_{2k}] \geq p$.
\end{lemma}
\begin{proof}
It clearly suffices to prove the result when $X$ is non-lazy.  Assume that $|X(t)| = j \in \{k,\ldots,2k\}$.  It is obvious that the probability that $|X|$ moves to $j+1$ in its next step is at least $1-\tfrac{2k}{d}$.  The reason is that the probability that the next coordinate to change is one of the coordinates of $X(t)$ whose value is $0$ is at least $1-\tfrac{2k}{d}$.  Similarly, the probability that $|X|$ next moves to $j-1$ is at most $\tfrac{2k}{d}$.  Consequently, the first result of the lemma follows from the Gambler's ruin problem (see, for example, \cite[Section 17.3.1]{LPW08}).  The second assertion of the lemma follows from the same argument.
\end{proof}

\begin{lemma}
\label{lem::hitting_bound}
Assume that $k \in \N$ and that $d = 2k \vee 3$.  Suppose that $X$ is a lazy random walk on $\Z^d$ and that $|X(0)| = 2k$.  Let $\tau_k$ be the first time $t$ that $|X(t)|=k$.  There exists $p_k > 0$ depending only on $k$ such that $\p[\tau_k = \infty] \geq p_k > 0$.
\end{lemma}
\begin{proof}
Let $\p_y$ denote the law under which $X$ starts at $y$.  Assume that $\p_y[\tau_k = \infty] = 0$ for some $y \in \Z^d$ with $|y|=2k$.  Suppose that $z \in \Z^d$ with $|z|=2k$ and let $\tau_z$ be the first time that $X$ hits $z$. Then since $\p_y[\tau_z < \tau_k] > 0$, it follows from the strong Markov property that $\p_z[\tau_k=\infty]=0$.  From this, it follows that the expected amount of time that $X$ spends in $B(0,k)$ is infinite because it implies that on each successive hit to $\partial B(0,2k)$, $X$ returns to $B(0,k)$ with probability $1$.  Since $X$ is transient \cite[Theorem 4.3.1]{LAW_LIM_10}, the expected amount of time that $X$ spends in $B(0,k)$ is finite.  This is a contradiction.
\end{proof}

\begin{lemma}
\label{lem::return_bound}
Assume that $k \in \N$ and $d \geq 2k \vee 3$.  Suppose that $X$ is a lazy random walk on $\Z_n^d$ and that $|X(0)| = 2k$.  Let $\tau_k$ be the first time $t$ that $|X(t)|=k$.  There exists $p_k, c_k > 0$ depending only on $k$ such that $\p[\tau_k > c_k d n^2] \geq p_k > 0$.
\end{lemma}
\begin{proof}
We first assume that $d=2k \vee 3$.  It follows from Lemma~\ref{lem::hitting_bound} that there exists a constant $p_{k,1} > 0$ depending only on $k$ such that $\p[\tau_k > \tau_{n/4}] \geq p_{k,1}$.  The local central limit theorem (see \cite[Chapter 2]{LAW_LIM_10}) implies that there exists constants $c_{k,1}, p_{k,2} > 0$ such that the probability that a random walk on $\Z_n^d$ moves more than distance $\tfrac{n}{4}$ in time $c_{k,1} n^2$ is at most $1-p_{k,2}$.  Combining implies the result for $d=2k \vee 3$.

Now we suppose that $d \geq 2k \vee 3$.  Let $(X_1(t),\ldots,X_d(t))$ be the coordinates of $X(t)$.  By re-ordering if necessary, we may assume without loss of generality that $X_{2k+1}(0),\ldots,X_d(0) = 0$.  Let $Y(t) = (X_1(t),\ldots,X_{2k}(t))$.  Then $Y$ is a random walk on $\Z_n^{2k}$.  Clearly, $|Y(0)| = 2k$ because $X(0)$ cannot have more than $2k$ non-zero coordinates.  For each $j$, let $\tau_j^Y$ be the first time $t$ that $|Y(t)| = j$.  Then $\tau_k^Y \leq \tau_k$.  For each $t$, let $N_t$ denote the number of steps that $X$ takes in the time interval $\{1,\ldots,t\}$ in which one of its first $2k$ coordinates is changed (in other words, $N_t$ is the number of steps taken by $Y$).  The previous paragraph implies that $\p[ N_{\tau_k^Y} \geq c_{k,1} n^2] \geq p_{k,3} > 0$ for a constant $p_{k,3} > 0$ depending only on $k$.  Since the probability that the first $2k$ coordinates are changed in any step is $k/d$ (recall that $X$ is lazy), the final result holds from a simple large deviations estimate.
\end{proof}

\begin{figure}[ht!]
\begin{center}
\includegraphics[width=0.7\textwidth]{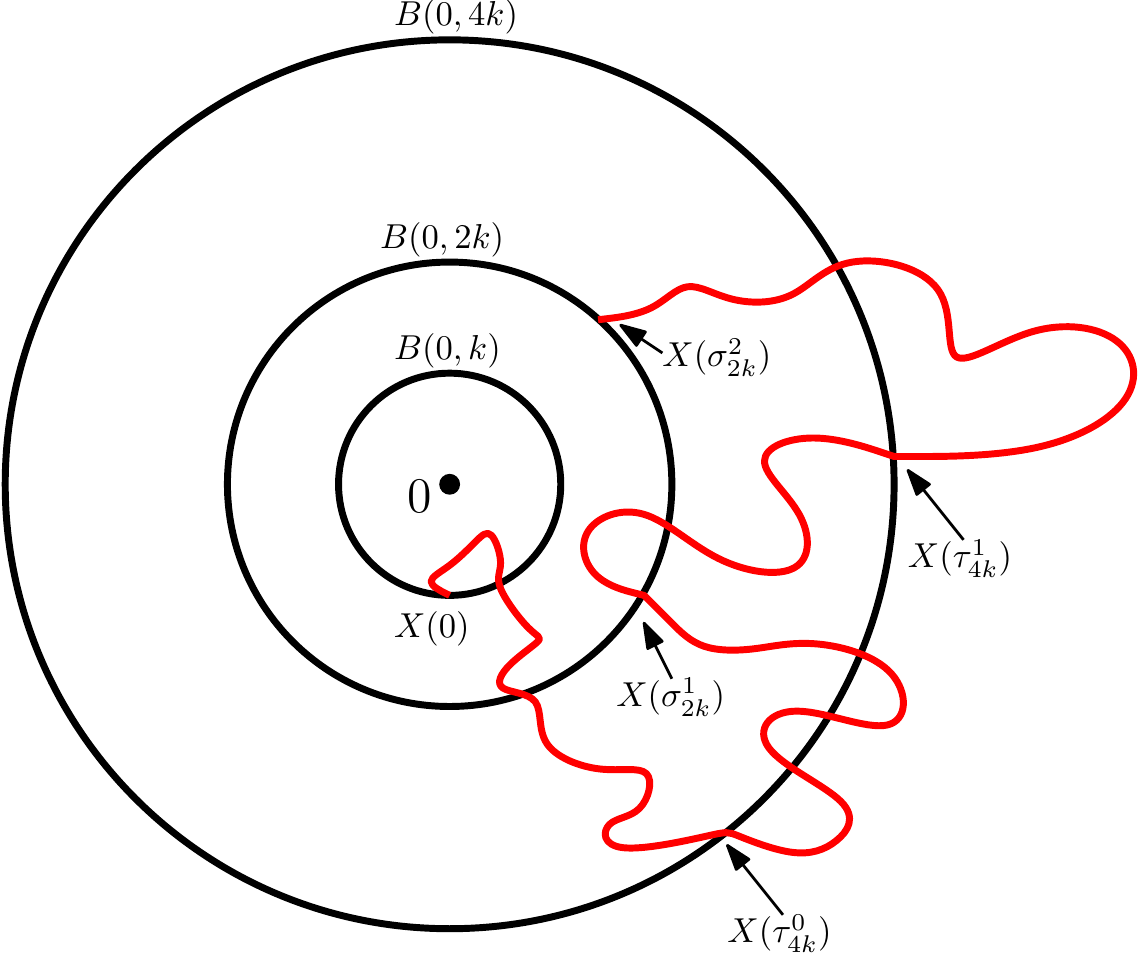}
\end{center}
\caption{\label{fig::excursions}  Assume that $d \geq 8$ and that $k \in \N$ with $d \geq 8k$.
Let $X$ be a lazy random walk on $\Z_n^d$ and that $X(0) = x$ with $|x-y| = k$.  In Proposition~\ref{prop::high_degree_green_bound}, we show that $G(x,y) \leq C_k d^{-k}$ where $C_k > 0$ is a constant depending only on $k$ .  By translation, we may assume without loss of generality that $|x|=k$ and $y=0$.  The idea of the proof is to first invoke Lemma~\ref{lem::escape_bound} to show that $X$ escapes to $\partial B(0,4k)$ with probability at least $1-C_{k,1} d^{-k}$.  We then decompose the path of $X$ into successive excursions $\{X(\sigma_{2k}^j),\ldots,X(\tau_{4k}^j),\ldots,X(\sigma_{2k}^{j+1})\}$ between $\partial B(0,2k)$ back to itself through $\partial B(0,4k)$.  By Lemma~\ref{lem::escape_bound}, we know that each excursion hits $0$ with probability bounded by $C_{2k,1} d^{-2k}$ and Lemma~\ref{lem::return_bound} implies that each excursion takes length $c_k d n^2$ with probability at least $p_k > 0$.  Consequently, the result follows from a simple stochastic domination argument.
}
\end{figure}

Now we are ready to prove our estimate of $G(x,y)$ when $d$ is large.

\begin{proposition}
\label{prop::high_degree_green_bound}
Suppose that $d \geq 8$.  Let $G(x,y)$ denote the Green's function for lazy random walk on $\Z_n^d$.  For each $k \in \N$ with $k \leq \tfrac{d}{8}$, there exists a constant $C_k > 0$ which does not depend on $n,d$ such that
\[ G(x,y) \leq \frac{C_k}{d^{k}} \text{ for all } x,y \in \Z_n^d \text{ with } |x-y| \geq k.\]
\end{proposition}
\begin{proof}
See Figure~\ref{fig::excursions} for an illustration of the proof.  By translation, we may assume without loss of generality that $y=0$; let $k=|x|$.  Let $\tau_0$ be the first time $t$ that $|X(t)| = 0$.  The strong Markov property implies that
\[ G(x,y) \leq \p[\tau_0 > \tu] + (1-\p[\tau_0 > \tu]) G(x,x).\]
Consequently, it suffices to show that for each $k \in \N$, there exists constants $C_k, C_0 > 0$ such that
\begin{align}
\p[\tau_0 > \tu] &\leq \frac{C_k}{d^{k}} \text{ and } \label{eqn::hit_before_uniform}\\
G(x,x) &\leq C_0 \label{eqn::green_diagonal}.
\end{align}
We will first prove \eqref{eqn::hit_before_uniform}; the proof of \eqref{eqn::green_diagonal} will be similar.

Let $N$ be a geometric random variable with success probability $C_{2k} d^{-2k}$ where $C_{2k}$ is the constant from Lemma~\ref{lem::escape_bound}.  Let $(\xi_j)$ be a sequence of independent random variables with $\p[\xi_j = c_{2k} d n^2] = p_{2k}$ and $\p[\xi_j = 0] = 1-p_{2k}$ where $c_{2k},p_{2k}$ are the constants from Lemma~\ref{lem::return_bound} independent of $N$.  We claim that $\tau_0$ is stochastically dominated from below by $\sum_{j=1}^{N \zeta -1} \xi_j$ where $\zeta$ is independent of $N$ and $(\xi_j)$ with $\p[ \zeta = 0] = C_k d^{-k} = 1-\p[\zeta=1]$.  Indeed, to see this we let $\sigma_k^0=0$ and let $\tau_{4k}^0$ be the first time $t$ that $|X(t)| = 4k$.  For each $j \geq 1$, we inductively let $\sigma_{2k}^j$ be the first time $t$ after $\tau_{4k}^{j-1}$ that $|X(t)|=2k$ and let $\tau_{4k}^j$ be the first time $t$ after $\sigma_{2k}^j$ that $|X(t)|=4k$.  Let $\CF_t$ be the filtration generated by $X$.  Lemma~\ref{lem::escape_bound} implies that the probability that $X$ hits $0$ in $\{\sigma_{2k}^j,\ldots,\tau_{4k}^j\}$ given $\CF_{\sigma_{2k}^j}$ is at most $C_{2k} d^{-2k}$ for each $j \geq 1$ where $C_{2k} > 0$ only depends on $2k$.  This leads to the success probability in the definition of $N$ above.  The factor $\zeta$ is to take into account the probability that $X$ reaches distance $2k$ before hitting $0$.  Moreover, Lemma~\ref{lem::return_bound} implies that $\p[ \sigma_{2k}^j - \tau_{4k}^{j-1} \geq c_{2k} d n^2 | \CF_{\tau_{4k}^j}] \geq p_{2k}$.  This leads to the definition of the $(\xi_j)$ above.  This implies our claim.

To see \eqref{eqn::hit_before_uniform} from our claim, an elementary calculation yields that
\[ \p[ N \zeta \leq C_{2k}^{-1} d^{k}] \leq \p[ N \leq C_{2k}^{-1} d^{k} \text{ or } \zeta = 0] \leq 2d^{-k} + C_k d^{-k}.\]
We also note that
\[\p\left[  \sum_{j=1}^{m} \xi_j \leq \frac{p_k c_k m n^2}{2} \right] \leq e^{-c m}\]
for some constant $c > 0$.  Combining these two observations along with a union bound implies \eqref{eqn::hit_before_uniform}.  To see \eqref{eqn::green_diagonal}, we apply a similar argument using the second assertion of Lemma~\ref{lem::escape_bound}.
\end{proof}

Now that we have proved Proposition~\ref{prop::low_degree_green_bound} and Proposition~\ref{prop::high_degree_green_bound}, we are ready to check the criteria of Assumption~\ref{assump::main}.

\subsection{Part \eqref{assump::main::hitting}}

By \cite[Proposition 1.14]{LPW08} with $\tau_x^+ = \min\{t \geq 1 : X(t) = x\}$, we have that $\E_x[\tau_x^+] = |\Z_n^d|$.  Applying Proposition~\ref{prop::high_degree_green_bound}, we see that there exists constants $d_0,r > 0$ such that if $d \geq d_0$, then
\begin{equation}
\label{eqn::green_half_bound}
G(x,y) \leq 1/2 \text{ for all } |x-y| \geq r.
\end{equation}
Proposition~\ref{prop::low_degree_green_bound} implies that there exists $n_0$ such that if $n \geq n_0$ and $3 \leq d < d_0$ then \eqref{eqn::green_half_bound} likewise holds, possibly by increasing $r$ (clearly, part~\eqref{assump::main::hitting} holds when $d \leq d_0$ and $n \leq n_0$; note also that we may assume without loss of generality that $d_0,n_0$ are large enough so that the diameter of the graph is at least $2r$).  Let $\tau_r$ be the first time $t$ that $|X(t) - X(0)| =r$.  We observe that there exists $\rho_0 = \rho_0(r) > 0$ such that
\begin{equation}
\label{eqn::move_away_bound}
\p_x[\tau_r < \tau_x^+] \geq \rho_0
\end{equation}
uniform in $n,d$ since in each time step there are $d$ directions in which $X(t)$ increases its distance from $X(0)$.  By combining \eqref{eqn::green_half_bound} with \eqref{eqn::move_away_bound}, we see that $\p_x[\tau_x^+ \geq \tu(\CG)] \geq \rho_1 > 0$ uniform $d \geq d_0$.  Let $\CF_t$ be the filtration generated by $X$.  We consequently have that
\begin{align*}
   \E_x[\tau_x^+]
&\geq \E_x[ \tau_x^+ \one_{\{ \tau_x^+ \geq \tu(\CG)\}}]
 = \E_x\big[ \E_x[ \tau_x^+ | \CF_{\tu(\CG)}] \one_{\{ \tau_x^+ \geq \tu(\CG)\}} \big]\\
&\geq \E_x\big[ \E_{X(\tu(\CG))}[\tau_x] \one_{\{\tau_x^+ \geq \tu(G)\}} \big]
 \geq \rho_1 \left(1-\frac{1}{2e}\right) \E_\pi[\tau_x].
\end{align*}
That is, there exists $\rho_2 > 0$ uniform in $d \geq d_0$ such that $\E_x[\tau_x^+] \geq \rho_2 \E_\pi[\tau_x]$.  Hence by \cite[Lemma 10.2]{LPW08}, we have that $\thit(\Z_n^d) \leq K_1|\Z_n^d|$ where $K_1 = 2/\rho_2$ is a uniform constant.

\begin{remark}
There is another proof of Part~\ref{assump::main::hitting} which is based on eigenfunctions.  In particular, we know that
\[ t_\hit(\Z_n^d) \leq 2 \E_\pi[\tau_x] = 4 \sum_i \frac{1}{1-\lambda_i}\]
where the $\lambda_i$ are the eigenvalues of simple random walk on $\Z_n^d$ distinct from $1$; the extra factor of $2$ in the final equality accounts for the laziness of the chain.  The $\lambda_i$ can be computed explicitly using \cite[Lemma 12.11]{LPW08} and the form of the $\lambda_i$ when $d=1$ which are given in \cite[Section 12.3]{LPW08}.  The assertion follows by performing the summation which can be accomplished by approximating it by an appropriate integral.
\end{remark}

\subsection{Part \eqref{assump::main::green_exponent}}

It follows from Proposition~\ref{prop::high_degree_green_bound} that there exist constants $C > 0$ and $d_0 \geq 3$ such that
\begin{equation}
\label{eqn::green_return_bound}
G(x,y)\leq \frac{C}{d} \text{ for } x,y \in \Z_n^d \text{ with } |x-y|=1
\end{equation}
provided $d \geq d_0$.  Consequently, there exists $K \in \N$ which does not depend on $d \geq d_0$ such that
\begin{equation}
\label{eqn::cycle_green_power_bound}
2K (5/2)^K G^K(x,y) = O\left( 2K \left(\frac{5/2}{d}\right)^K \right)
\end{equation}
It follows by combining \eqref{eqn::relaxation_asymptotics} and \eqref{eqn::uniform_asymptotics} that we have that
\begin{equation}
\label{eqn::cycle_uniform_relax_ratio}
\frac{\tu(\Z_n^d)}{\trel(\Z_n^d)} = O(\log d).
\end{equation}
Combining \eqref{eqn::cycle_green_power_bound} with \eqref{eqn::cycle_uniform_relax_ratio} shows that part \eqref{assump::main::green_exponent} of Assumption~\ref{assump::main} is satisfied provided we take $K_2 = K$ large enough.  Moreover, \eqref{eqn::cycle_uniform_relax_ratio} clearly holds if $3 \leq d < d_0$ by Proposition~\ref{prop::low_degree_green_bound}.

\subsection{Part \eqref{assump::main::green_ratio}}

We first note that it follows from \eqref{eqn::relaxation_asymptotics}, \eqref{eqn::uniform_asymptotics}, Proposition~\ref{prop::low_degree_green_bound}, and Proposition~\ref{prop::high_degree_green_bound} that there exists constants $C > 0$ such that $n^*$ for $\Z_n^d$ is at most $C d^2 n^2 \log d$ for all $d \geq 3$.  To check this part, we need to show that there exists $K_3 > 0$ such that
\begin{equation}
\label{eqn::gstar_goal}
 G^*(n^*) \leq K_3\left(\frac{dn^2 + d \log n}{ \log d + \log n}\right).
\end{equation}

We are going to prove the result by considering the regimes of $d \leq \sqrt{\log n}$ and $d > \sqrt{\log n}$ separately.
$\ $\newline

\noindent{\bf Case 1:} $d < \sqrt{\log n}$.

From \eqref{eqn::gstar_goal} it is enough to show that $G^*(n^*) \leq K d n^2 / \log n$.  We can bound $G^*(n^*)$ in this case as follows.  Let $D = (d \log d \log n)^{1/(\tfrac{1}{2}d-1)}$.  By Proposition~\ref{prop::low_degree_green_bound}, we can bound from above the expected amount of time that $X$ starting at $0$ in $\Z_n^d$ spends in the $L^1$ ball of radius $D$ by summing radially:
\begin{align*}
        & \sum_{k=1}^D \frac{C_1}{d} \left(\frac{4d}{\pi}\right)^{d/2} k^{1-d/2} \cdot 2d (2k)^{d-1}\\
 \leq& C_1 \left(\frac{16d}{\pi}\right)^{d/2} \sum_{k=1}^D  k^{d/2}
 \leq \frac{C_2}{d}\left(\frac{16 d}{\pi}\right)^{d/2} D^{1+d/2} \leq C_3 n(d \log d \log n)^5
\end{align*}
for constants $C_1,C_2,C_3 > 0$, where we used that $d^{d/2} \leq n$.  We also note that $2d(2k)^{d-1}$ is the size of the $L^\infty$ ball of radius $k$.  The exponent of $5$ comes from the inequality
\[ \frac{\tfrac{1}{2} d + 1}{\tfrac{1}{2}d-1} \leq 5 \text{ for all } d \geq 3.\]
We can estimate $G^*(n)$ by dividing between the set of points which have distance at most $D$ to $0$ and those whose distance to $0$ exceeds $D$ by:
\begin{align*}
     G^*(n^*)
\leq& C_3 n (d \log d \log n)^5 + C_4 D^{1-\tfrac{1}{2}d} n^*\\
\leq&  C_3 n (d \log d \log n)^5 + \frac{C_4 \cdot C d^2 n^2 \log d}{d \log d \log n},
\end{align*}
where $C_4 > 0$ is a constant and we recall that $C > 0$ is the constant from the definition of $n^*$.
This implies the desired result.
$\ $\newline

\noindent{\bf Case 2:} $d \geq \sqrt{\log n}$.

In this case, we are going to employ Proposition~\ref{prop::high_degree_green_bound} to bound $G^*(n^*)$.  The number of points which have distance at most $k$ to $0$ is clearly $1+(2d)^k$.  Consequently, by Proposition~\ref{prop::high_degree_green_bound}, we have that
\begin{align*}
G^*(n^*) \leq& \left(C_0 +  \sum_{k=1}^3 C_k d^{-k}(2d)^k \right)  + C_4 d^{-4}  n^*\\
   \leq& C_5 +  \frac{C_6 (\log d)  n^2}{d^2}
\end{align*}
for some constants $C_5, C_6 > 0$.  Since $d^2 \geq \log n$, this is clearly dominated by the right hand side of \eqref{eqn::gstar_goal} (with a large enough constant), which completes the proof in this case.

\qed

\section{Coverage Estimates}
\label{sec::estimates}

Throughout, we assume that $\CG$ is a finite, connected, vertex transitive graph and $X$ is lazy random walk on $\CG$ with transition matrix $P$ and stationary measure $\pi$.  For $S \subseteq V(\CG)$, we let $\CC_S(t)$ be the set of vertices in $S$ visited by $X$ by time $t$ and let $\CU_S(t) = S \setminus \CC_S(t)$ be the subset of $S$ which $X$ has not visited by time $t$.  We let $\CC(t) = \CC_{V(\CG)}(t)$ and $\CU(t) = \CU_{V(\CG)}(t)$.  We will use $\p_x,\E_x$ to denote the probability measure and expectation under which $X(0) = x$.  Likewise, we let $\p_\pi,\E_\pi$ correspond to the case that $X$ is initialized at stationarity.  The purpose of this section is to develop a number of estimates which will be useful for determining the amount of time required by $X$ in order to cover subsets $S$ of $V(\CG)$.  We consider two different regimes depending on the size of $S$.  If $S$ is large, we will estimate the amount of time it takes for $X$ to visit $\tu(\CG)$ distinct vertices in $S$.  If $S$ is small, we will estimate the amount of time it takes for $X$ to visit $1/2$ of the vertices in $S$.

\subsection{Large Sets}
\label{subsec::large_estimates}

In this subsection, we will prove that the amount of time it takes for $X$ to visit $\tu(\CG)$ distinct elements of a large set of vertices $S \subseteq V(\CG)$ is stochastically dominated by a geometric random variable whose parameter depends on $\tu(\CG)/\trel(\CG)$.  The main result is:

\begin{proposition}
\label{prop::start}
Assume $X$ satisfies part \eqref{assump::main::green_exponent} of Assumption~\ref{assump::main} with constants $K_2,K_3$.  Let $S \subseteq V(\CG)$ consist of at least $2K_2 \tu(\CG) / G(x,y)$ elements for $x,y \in V(\CG)$ adjacent and let
\[ t=\frac{2(K_2+2) \tu(\CG)}{\pi(S)}.\]
There exists a universal constant $C > 0$ such that for every $x \in V(\CG)$, we have that
\[ \p_x\left[\CC_S(t) \leq \tu(\CG) \right] \leq \exp\left(- C \frac{\tu(\CG)}{\trel(\CG)} \right).\]
\end{proposition}

Recall that
\[ \CL_S(t) = \sum_{s=0}^t \ind_{\{X(s) \in S\}}\]
is the amount of time that $X$ spends in $S$ up to time $t$.  The proof consists of several steps.  The first is Proposition~\ref{prop::local_time}, which we will deduce from \cite[Theorem 1]{LP04} shortly, which gives that the probability $\CL_S(t)$ is less than $1/2$ its mean is exponentially small in $t$.  Once we show that $\CL_S(t)$ is large with high probability, in order to show that $X$ visits many vertices in $S$, we need to rule out the possibility of $X$ concentrating most of its local time in a small subset of $S$.  This is accomplished in Lemma~\ref{lem::deconc}.  We now proceed to the proof of Proposition~\ref{prop::local_time}.

\begin{proof}[Proof of Proposition~\ref{prop::local_time}]
We rewrite the event
\begin{equation}
\label{eqn::event}
 \left\{ \CL_S(t) \leq  t \frac{\pi(S)}{2}\right\}= \left\{ \sum_{s=0}^t f(X_s) > t\left(1-\pi(S) + \frac{\pi(S)}{2}\right)\right\}
\end{equation}
where $f(x)= \ind_{S^c}(x)$.
Let $\epsilon = \pi(S)/2$ and $\mu = \E_\pi[f(X(t))] = 1-2\epsilon$.  The case $\epsilon \geq 1/4$ follows immediately from \cite[Equation 3]{LP04} in the statement of \cite[Theorem 1]{LP04}, so we will only consider the case $\epsilon \in (0,1/4)$ here.
Let $\ol{\mu} = 1-\mu = 2\epsilon$.  For $x \in (0,1)$, let
\[ I(x) = -x \log\left( \frac{\mu + \ol{\mu} \lambda_0}{1-2\ol{x}/(\sqrt{\Delta} + 1)} \right) - \ol{x} \log \left( \frac{\ol{\mu} + \mu \lambda_0}{1-2x /( \sqrt{\Delta}+1)} \right)\]
where $\ol{x} = 1-x$ and
\begin{equation}
\label{eqn::delta_definition}
 \Delta = 1 + \frac{4\lambda_0 x \ol{x}}{\mu \ol{\mu}(1-\lambda_0)^2}.
\end{equation}
For $x \in [\mu,\mu+\epsilon] = [1-2\epsilon,1-\epsilon]$, $\epsilon \in (0,1/4)$, and $\lambda_0 \geq 1/2$, we note that
\begin{equation}
\label{eqn::delta_bound}
 \frac{1/2}{ (1-\lambda_0)^2} \leq \Delta \leq \frac{20}{(1-\lambda_0)^2}
\end{equation}
By \cite[Theorem 1]{LP04} and using the representation \eqref{eqn::event}, we have that
\[ \p_\pi\left[\CL_S(t) \leq t \epsilon \right] \leq \exp(- I(\mu+\epsilon) t).\]
Since $I(\mu) = I'(\mu) = 0$ and $I''(x) = (\sqrt{\Delta} x\ol{x})^{-1}$ (see \cite[Appendix B]{LP04}), we can write
\begin{equation}
\label{eqn::rate_function_integral}
 I(\mu+\epsilon) = \int_\mu^{\mu+\epsilon} \int_\mu^x \frac{1}{\sqrt{\Delta} y \ol{y}} dy dx
\end{equation}
where $\ol{y} = 1-y$.  Inserting the bounds from \eqref{eqn::delta_bound}, we thus see that the right side of \eqref{eqn::rate_function_integral} admits the lower bound
\[ \frac{1-\lambda_0}{\sqrt{20}} \int_{1-2\epsilon}^{1-\epsilon} \int_{1-2\epsilon}^x \frac{1}{2\epsilon} dy dx \geq \frac{(1-\lambda_0)\epsilon}{16\sqrt{5}} \]
for all $\epsilon \in (0,1/4)$ and $\lambda_0 \geq \tfrac{1}{2}$.
\end{proof}

As in the proof of Lemma~\ref{lem::lazy_transition}, we couple $X$ with a non-lazy random walk $Y$ so that $X(t) = Y(N_t)$ where $N_t = \sum_{i=0}^t \xi_i$ and the $(\xi_i)$ are iid with $\p[\xi_i = 0] = \p[\xi_i = 1] = \tfrac{1}{2}$ and are independent of $Y$.  We let $\CL_S^Y(t)$ denote the amount of time that $Y|_{[0,N_t]}$ spends in $S$ (note that this differs slightly from the definition of $\CL_x^Y(t)$ which appeared in Section~\ref{sec::examples}).  In other words, $\CL_S^Y$ is the amount of that $X$ spends in $S$ by time $t$, not including those times where $X$ does not move.  The next lemma gives a lower bound on the probability that the number $\CC_S(t)$ of distinct vertices $X$ visits in a given set $S \subseteq V(\CG)$ by time $t$ is proportional to $\CL_S^Y(t)$.  The lower bound for this probability will be given in terms of the Green's function $G(x,y)$ for $X$.  Recall its definition from \eqref{eqn::green_definition}.  Since $X$ is a lazy random walk, we also have that
\begin{equation}
\label{eqn::green_function_diagonal_bound}
 G(x,y) \leq G(x,x) \text{ for all } x,y \in V(\CG).
\end{equation}
This is a consequence of \eqref{eqn::lazy_monotone_distance}.

\begin{lemma}
\label{lem::deconc}
Fix $S \subseteq V(\CG)$.  For each positive integer $k$, we have that
\begin{equation}
 \p_\pi \left[ \CC_S(t) \geq  \frac{\CL_S^Y(t) - \tu(\CG)}{k} \right]  \geq  1- \frac{t \pi(S)q^k(t)}{\tu(\CG)},
\end{equation}
where
\begin{equation}
\label{eqn::q_definition}
 q(t) = (G(x,y)-1)+(1+(2e)^{-1})\frac{t-\tu(\CG)}{|\CG|}\ind_{\{t>\tu(\CG)\}}.
\end{equation}
and $x$ is adjacent to $y$.
\end{lemma}
\begin{proof}
For $t \geq \tu(\CG)$, we have $P^t(x,y) \leq (1+(2e)^{-1})\pi(y)$ by the definition of $\tu(\CG)$.
Thus by a union bound,
\[ \p_x[\CL_x^Y(t) > 1] \leq q(t).\]
Hence by the strong Markov property,
\[ \p_x[\CL_x^Y(t) > k] \leq q^k(t).\]
Observe
\begin{equation}
\label{eqn::simple_hitting_bound}
\p_\pi[\tau_x = s] \leq \p_\pi[X_s = x] \leq \pi(x).
\end{equation}
Let
\[ \CL_{S,k}^Y(t) = \sum_{x \in S} \CL_x^Y(t) \one_{\{\CL_x^Y(t) > k\}}\]
be the total time that $Y$ spends at points in $S$ which it visits more than $k$ times by time $N_t$.  By \eqref{eqn::simple_hitting_bound}, we have that
\begin{align*}
      \E_\pi[\CL_{S,k}^Y(t)]
 \leq \sum_{x \in S} \sum_{s=0}^{t} \p_\pi[\tau_x=s] q^k(t)
 \leq t \pi(S) q^k(t).
\end{align*}
Applying Markov's inequality we have that
\begin{align*}
  \p_\pi\left[ \CL_{S,k}^Y(t) \geq \tu(\CG)  \right]
 &\leq \frac{\E_\pi[ \CL_{S,k}^Y(t)]}{\tu(\CG)}
 \leq \frac{t \pi(S) q^k(t)}{\tu(\CG)}
\end{align*}
Observe
\begin{align*}
  \CC_S(t)
= \sum_{x \in S} \one_{\{\CL_x^Y(t) \geq 1\}}
\geq \sum_{x \in S} \big(\one_{\{\CL_x^Y(t) \geq 1\}} - \one_{\{\CL_x^Y(t) > k\}}\big)
 \geq \frac{\CL_S^Y(t) - \CL_{S,k}^Y(t)}{k}.
\end{align*}
Thus
\[ \left\{ \CL_{S,k}^Y(t) < \tu(\CG) \right\} \subseteq \left\{ \CC_S(t) \geq \frac{\CL_S^Y(t) - \tu(\CG)}{k} \right\}.\]
We arrive at
\begin{align*}
      \p_\pi\left[\CC_S(t) \geq \frac{\CL_S^Y(t) - \tu(\CG)}{k}\right]
&\geq 1-\p_\pi\left[ \CL_{S,k}^Y(t) \geq \tu(\CG) \right]
 \geq 1- \frac{t \pi(S) q^k(t)}{\tu(\CG)},
\end{align*}
which completes the proof of the lemma.
\end{proof}

Proposition~\ref{prop::local_time} gives a lower bound on the probability $\CL_S(t)$ is proportionally lower than its expectation, Lemma~\ref{lem::deconc} gives a lower bound on the probability $X$ visits less than a positive fraction of $\CL_S^Y(t) - \tu(\CG)$ vertices in $S$ by time $t$, and standard large deviations estimates bound the probability that $\CL_S^Y(t)$ is proportionally smaller than $\CL_S(t)$.  By combining these two lemmas, we obtain the following result, which gives a lower bound on the rate at which $X$ covers vertices in $S$.
\begin{lemma}
\label{lem::main}
Fix $S \subseteq V(\CG)$. Then
\begin{align}
 \label{eqn::coverage_lower_bound}
&\p_\pi\left[ \CC_S(t) \leq \frac{t\pi(S) - 8\tu(\CG)}{8k} \right]\\
 &\ \ \ \ \ \ \ \ \leq \exp\left(-C_0 t\frac{\pi(S)}{\trel(\CG)}\right) + \exp\left(-\frac{1}{16} t \pi(S)\right)  + \frac{t\pi(S) q^k(t)}{\tu(\CG)} \notag
\end{align}
where the constant $C_0$ is as in Proposition~\ref{prop::local_time} and the function $q$ is as in \eqref{eqn::q_definition}.
\end{lemma}
\begin{proof}
We trivially have that
\begin{align*}
 &\p_\pi\left[ \CC_S(t) \geq \frac{t\pi(S) - 8\tu(\CG)}{8k} \right]
\geq \p_\pi\left[ \CC_S(t) \geq \frac{t\pi(S) - 8\tu(\CG)}{8k},\ \CL_S^Y(t) > \frac{t}{8} \pi(S)\right]\\
\geq& \p_\pi\left[ \CC_S(t) > \frac{\CL_S^Y(t) - \tu(\CG)}{k},\ \CL_S^Y(t) > \frac{t}{8} \pi(S)\right].
\end{align*}
Therefore
\begin{align*}
\p_\pi\left[ \CC_S(t) \leq \frac{t\pi(S) - 8\tu(\CG)}{8k} \right] \leq
\p_\pi\left[ \CL_S^Y(t) \leq \frac{t}{8} \pi(S)\right] + \p_\pi\left[ \CC_S(t) \leq \frac{\CL_S^Y(t) - \tu(\CG)}{k} \right]
\end{align*}
We can bound the second term from above by Lemma~\ref{lem::deconc}.  The first term is bounded from above by
\[ \p_\pi\left[ \CL_S^Y(t) \leq \frac{t}{8} \pi(S)\right]
  \leq \p_\pi\left[ \CL_S(t) \leq \frac{t}{2} \pi(S)\right] +
		\p_\pi\left[ \CL_S(t) > \frac{t}{2} \pi(S),\ \ \CL_S^Y(t) \leq \frac{t}{8} \pi(S)\right].\]
We can bound the first term using Proposition~\ref{prop::local_time}.  Conditionally on $\{\CL_S(t) > \tfrac{t}{2}\pi(S)\}$, we note that $\{\CL_S^Y(t) \leq \tfrac{t}{8}\pi(S)\}$ occurs if $X$ stays in place for at least $\tfrac{3t}{8}\pi(S)$ time steps.  Consequently, standard large deviations estimates imply that the second term above is bounded by $\exp(-\tfrac{1}{16} t \pi(S))$.
\end{proof}

We can now easily complete the proof of Proposition~\ref{prop::start} by ignoring the first $\tu(\CG)$ units of time in order to reduce to the stationary case, then apply Assumption~\ref{assump::main} in order to match the error terms in Lemma~\ref{lem::main}.

\begin{proof}[Proof of Proposition~\ref{prop::start}]
We first observe that
\[ \p_x\left[\CC_S(t) \leq \tu(\CG) \right] \leq (1+(2e)^{-1}) \p_{\pi}[ \CC_S(t-\tu(\CG)) \leq \tu(\CG)].\]

With $\wt{t}= 2 K_2 \tu(\CG)/ \pi(S)$ and using $|S| \geq 2 K_2 \tu(\CG) / (G(x,y)-1)$ for $x,y \in V(\CG)$ adjacent, we see that
\[G(x,y) - 1 \leq q(\wt{t})\leq \frac{5}{2}(G(x,y)-1).\]
Combining this with part \eqref{assump::main::green_exponent} of Assumption~\ref{assump::main} implies
\begin{equation}
 \label{errorprob} \frac{\wt{t} \pi(S) q^{K_2}(\wt{t})}{\tu(\CG)} \leq 2K_2 q^{K_2}(\wt{t}) \leq \exp\left(- \frac{\tu(\CG)}{\trel(\CG)} \right).
\end{equation}
Applying Lemma~\ref{lem::main} gives the result.
\end{proof}

\subsection{Small Sets}
\label{subsec::small_estimates}

We will now give an upper bound on the rate at which $X$ covers $1/2$ the elements of a set of vertices $S \subseteq V(\CG)$, provided $|S|$ is sufficiently small.

\begin{proposition}
\label{prop::cover_small_set}
Fix $S \subseteq V(\CG)$, let $s = |S|$, and assume that
\[ \tu(\CG) \leq \frac{|\CG|}{4s}.\]
There exists constants $C_2,C_3 > 0$ such that
\[ \p_x \left[ \CC_S(C_2 |\CG| G^*(s)) \leq \frac{s}{2} \right] \leq \exp(-C_3 s)\]
for all $x \in V(\CG)$.
\end{proposition}

The main step in the proof of Proposition~\ref{prop::cover_small_set} is the next lemma, which gives an upper bound on the hitting time for $S$.  Its proof is based on the following observation.  Suppose that $S \subseteq V(\CG)$ and $\tau_S = \min\{t \geq 0 : X(t) \in S\}$.  Let $Z$ be a non-negative random variable with $Z \one_{\{\tau_S > t\}} = 0$ and $\E_x[Z \one_{\{\tau_S \leq t\}}] > 0$.  Then we have that
\begin{equation}
\label{eqn::green_kernel_hitting_estimate}
 \p_x[ \tau_S \leq t] = \frac{\E_x[Z]}{\E_x[Z|\tau_S \leq t]}.
\end{equation}
We will take $Z$ to be the amount of time $X$ spends in $S$.

\begin{lemma}
\label{lem::hit_small_set}
Fix $S \subseteq V(\CG)$ and let $s = |S|$.  Assume that
\[ \tu(\CG) \leq \frac{|\CG|}{2s}.\]
There exists a universal constant $\rho_0 > 0$ such that $x \in V(\CG)$ we have
\[ \p_x\left[ \tau_S \leq \frac{|\CG|}{s}\right] \geq \frac{\rho_0}{G^*(s)}.\]
\end{lemma}
\begin{proof}
Let us introduce $E=\left\{\tau_S \leq \frac{|\CG|}{s}\right\} $. Observe that
\[ \p_x[ E] \geq \frac{\E_x\left[\CL_S\left(\frac{|\CG|}{s}\right)\right]}{\E_x\left[\CL_S\left(\frac{|\CG|}{s}\right)| E\right]}\]
We can bound the numerator from below as follows:
\begin{align}
 \E_x\left[ \CL_S\left(\frac{|\CG|}{s}\right) \right]
&\geq (1-(2e)^{-1}) \E_\pi\left[ \CL_S\left(\frac{|\CG|}{s} - \tu(\CG)\right) \right] \notag\\
&\geq (1-(2e)^{-1}) \pi(S) \left(\frac{|\CG|}{s}-\tu(\CG)\right) \geq \frac{1}{4}.
\label{eqn::num_hit_small_set}
\end{align}
Let $\CL_S(u,t) = \CL_S(t) - \CL_S(u-1)$ be the number of times in the set $\{u,\ldots,t\}$ that $X$ spends in $S$.  Then we can express the denominator as the sum
\begin{align*}
  &\E_x\left[\CL_S\left(\tau_S,\tau_S + \tu(\CG)\right) | E \right] + \E_x\left[\CL_S\left(\tau_S+\tu(\CG)+1,\frac{|\CG|}{s}\right) | E \right]\\
  =: &D_1 + D_2.
\end{align*}
We have
\[ D_2 \leq (1+(2e)^{-1}) \E_\pi\left[\CL_S\left(\frac{|\CG|}{s}\right) \right] \leq 2.\]
We will now bound $D_1$.  By the strong Markov property, we have that
\begin{align*}
	    D_1
&\leq \max_{z \in S} \E_z[\CL_S(\tu(\CG))]
  = \max_{z \in S} \E_z \sum_{t=0}^{\tu(\CG)} \one_{\{X(t) \in S\}} \\
&= \max_{z \in S} \sum_{y \in S} G(z,y)
 \leq G^*(s).
\end{align*}
Putting everything together completes the proof.
\end{proof}

The remainder of the proof of Proposition~\ref{prop::cover_small_set} is based on a simple stochastic domination argument.

\begin{proof}[Proof of Proposition~\ref{prop::cover_small_set}]
Let $C_2 > 0$; we will fix its precise value at the end of the proof.  That $X$ visits at least $s/2$ points in $S$ by the time $C_2 |\CG| G^*(s)$ with probability exponentially close to $1$ in $s$ follows from a simple large deviation estimate of a binomial random variable.  Namely, we run the chain for $C_2 G^*(s) s$ rounds, each of length $|\CG|/s$.  We let $S_0 = S$ and inductively let $S_i = S_{i-1} \setminus \{x\}$ if $X$ hits $x$ in the $i$th round for $i \geq 1$.  If $|S_i| \geq s/2$, the hypotheses of Lemma~\ref{lem::hit_small_set} hold.  In this case, the probability that $X$ hits a point in $S_i$ in the $i$th round is at least $\rho_0/G^*(s) > 0$.
Thus by stochastic domination, we have that
\[ \p[\CC_S(C_2 |\CG| G^*(s) )< s/2] \leq \p\left[ Z < s/2 \right] \]
where $Z\sim {\rm BIN}(C_2 G^*(s)s, \rho_0/G^*(s))$. By picking $C_2$ large enough ($C_2 > 1/\rho_0$ will do, say) and applying the Chernoff bound, we see that
\begin{equation}
\label{error2} \p\left[ \CC_S(C_2 |\CG| G^*(s)) < s/2\right] \leq \exp(-C_3 s)
\end{equation}
for some constant $C_3$ (one can check that $C_3=\tfrac{1}{8}$ suffices).  This estimate also holds if $s=1$. In this case we cover the point with constant probability in $C_2 |\CG|$ steps.
\end{proof}

\section{Proof of Theorem~\ref{thm::main}}
\label{sec::proof_of_main}

Throughout this section, we shall assume that $X$ is a lazy random walk on a graph $\CG$ which satisfies Assumption~\ref{assump::main}.  Recall that $\CU(t)$ is the set of vertices of $\CG$ which $X$ has not visited by time $t$.  We will use the notation $\p_x,\E_x$ for the probability measure and expectation under which $X(0) = x$.  Likewise, we let $\p_\pi,\E_\pi$ correspond to the case that $X$ is initialized at stationarity.  We will now work towards completing the proof of Theorem~\ref{thm::main} by applying the results of the previous section to describe the process by which $X$ covers $V(\CG)$.  We will study the process of coverage in two different regimes: before and after $\CU(t)$ contains at least $n^*$ vertices (recall the definition of $n^*$ from part~\eqref{assump::main::green_ratio} of Assumption~\ref{assump::main}).  To this end, we let
\begin{align*}
 r&=\max\{i: |\CG| - i\tu(\CG) \geq n^*\}, \\
\wt r&= \lfloor\log_2 (|\CG|- r \tu(\CG)) \rfloor
\end{align*}
and
\begin{align*}
  s_i&= |\CG|- i \tu(\CG), \quad i=0,\dots, r,\\
 s_{r+i}&=\left\lfloor \frac{s_r}{2^i} \right\rfloor \quad i=1,\dots, \wt r-1,\\
 s_{r+\wt r}&=0.
\end{align*}

We also define the stopping times
\[ T_i = \min\{ t \geq 1: \mathcal |\CU(t)|\le s_i\},\ \ i= 1,\dots, r+\wt r.\]

\begin{lemma}
\label{lem::large_decimation}
There exists constants $C_4,C_5$ such that for each $1 \leq i \leq r$ and all $x \in V(\CG)$, we have that
\begin{equation}
 \label{eqn::large_decimation}
 \p_x\left[|\CU(t)|> s_i\right]
 \leq
 \exp\left(\frac{s_i}{\trel(\CG)} \left(C_4 \log |\CG| - \frac{C_5}{|\CG|} t\right) \right).
\end{equation}
\end{lemma}
\begin{proof}
For each $i \in \{1,\ldots,r\}$, we let
\[ t_i = \frac{2(K_2+2) \tu(\CG) |\CG|}{s_i}\]
Proposition~\ref{prop::start} implies that
\[ \p_x[|\CU(t+t_i)| \leq s_{i+1}\ |\ |\CU(t)| \in (s_{i+1},s_i] ]
   \geq 1- \exp\left(- C \frac{\tu(\CG)}{\trel(\CG)} \right).\]
Consequently, it follows that there exists independent variables $Z_j \sim {\rm GEO}(1-\exp(-C  \tu(\CG)/\trel(\CG)))$ such that $T_j - T_{j-1}$ is stochastically dominated by $t_j Z_j$ for all $j \in \{1,\ldots,r\}$.  Thus for $\theta_i > 0$, we have that
\begin{align}
   \p_x[ |\CU(t)| > s_i]
&= \p_x[T_i > t] = \p_x\left[ \sum_{j=1}^i T_j - T_{j-1} > t \right] \notag\\
&\leq e^{-\theta_i t} \prod_{j=1}^i \E_x[ e^{\theta_i t_j Z_j}]. \label{eqn::large_decim_exp}
\end{align}
Note that for every $\beta \in (0,1)$ there exists $\alpha = \alpha(\beta) > 0$ such the moment generating function of a ${\rm GEO}(p)$ random variable satisfies
\begin{equation}
\label{eqn::geo_exp_moment}
 \frac{pe^x}{1-(1-p) e^x} \leq e^{\alpha x} \text{ provided } (1-p) e^x \leq \beta.
\end{equation}
Choosing
\[ \theta_i = \frac{C \tu(\CG)}{2 t_i \trel(\CG)}\]
we have that
\[ \theta_i t_j = \frac{C \tu(\CG)}{\trel(\CG)} \cdot \frac{t_j}{t_i} = \frac{C \tu(\CG)}{2 \trel(\CG)} \cdot \frac{s_i}{s_j}.\]
Hence as $s_i \leq s_j$ for all $i,j \in \{1,\ldots,r\}$ with $j \leq i$, we have
\[ \exp\left(\frac{C \tu(\CG)}{2 \trel(\CG)} \cdot \frac{s_i}{s_j} - \frac{ C \tu(\CG)}{\trel(\CG)} \right) \leq \exp\left(- \frac{C \tu(\CG)}{2 \trel(\CG)} \right) \leq \exp(-C/2).\]
Let $\alpha = \alpha(e^{-C/2})$ as in \eqref{eqn::geo_exp_moment}.  Consequently, we can bound the product of exponential moments in \eqref{eqn::large_decim_exp} by
\begin{align*}
\log \prod_{j=1}^i \E_x[e^{\theta_i t_j Z_j}]
 &\leq \alpha \sum_{j=1}^i \theta_i t_j
 =  \frac{\alpha C \tu(\CG) s_i}{2 \trel(\CG)} \sum_{j=1}^i \frac{1}{s_j}\\
 &=  \frac{\alpha C  s_i}{2 \trel(\CG)} \sum_{j=1}^i \frac{1}{|\CG|/\tu(\CG) - j}
 \leq \frac{\alpha C  s_i}{2\trel(\CG)} \log|\CG|.
\end{align*}
Inserting this expression into \eqref{eqn::large_decim_exp} gives \eqref{eqn::large_decimation}.
\end{proof}

\begin{lemma}
\label{lem::small_decimation}
There exists constants $C_6,C_7$ such that for all $1 \leq i \leq \wt{r}$ and $x \in V(\CG)$, we have that
\begin{align}
 \label{eqn::small_decimation}
 &\p_x\left[|\CU(t)| > s_{r+i} \right]\\
 \leq&
  \p_x[|\CU(t/2)| > s_r]
   + \exp\left(s_{r+i-1} \left(C_6 i - \frac{C_7}{|\CG| G^*(n^*)} t \right) \right). \notag
\end{align}
\end{lemma}
\begin{proof}
Let
\[ q_{r+j} = C_2 |\CG| G^*(s_{r+j})\]
where $C_2$ is as in Proposition~\ref{prop::cover_small_set}.  Proposition~\ref{prop::cover_small_set} implies that
\[ \p_x[ |\CU(t+q_{r+j})| \leq s_{r+j+1}\ |\ |\CU(t)| \in (s_{r+j+1},s_{r+j}]] \geq 1- \exp(-C_3 s_{r+j})\]
for $j \in \{1,\cdots, \wt{r}\}$.  Consequently, there exists independent random variables $Z_{r+j} \sim {\rm GEO}(1- \exp(-C_3 s_{r+j}))$ such that $T_{r+j}-T_{r+j-1}$ is stochastically dominated by $q_{r+j} Z_{r+j}$.  We have that
 \begin{align}
 &\p_x[ |\CU(t)| > s_{r+i}] = \p_x[T_{r+i}> t] \notag\\
 \leq& \p_x\left[T_r > \frac{t}{2}\right] + \p_x\left[ \sum_{j=1}^i T_{r+j}-T_{r+j-1}> \frac{t}{2}\right]
  =: I_1 + I_2 \label{eqn::small_dec_bound}
 \end{align}
Using that $I_1 = \p[|\CU(t/2)| > s_r]$ gives the first term in \eqref{eqn::small_decimation}.  We now turn to bound $I_2$.  Fixing $\theta_{r+i} > 0$, we have
 \begin{align}
  I_2
  &\leq e^{-\theta_{r+i} t/2} \prod_{j=1}^{i} \E_x\left[ e^{\theta_{r+i} q_{r+j} Z_{r+j}}\right]. \label{eqn::i2_exp_mom}
\end{align}
With the particular choice
\[ \theta_{r+i} = \frac{C_3}{2 C_2} \frac{s_{r+i}}{|\CG|G^*(n^*)}\]
we have that
\[ \exp( \theta_{r+i} q_{r+j} - C_3 s_{r+j}) \leq \exp(-C_3/2)=: \beta < 1.\]
Here, we used that if $n \leq m$ then $G^*(n) \leq G^*(m)$.  Thus by \eqref{eqn::geo_exp_moment} there exists $\alpha = \alpha(\beta) > 0$ such that we can bound the exponential moments in \eqref{eqn::i2_exp_mom} by
\begin{align*}
 \log \prod_{j=1}^i \E_x\left[ e^{\theta_{r+i} q_{r+j} Z_{r+j}}\right]
\leq \alpha \theta_{r+i} \sum_{j=1}^i q_{r+j}
= \frac{\alpha C_3}{2}  i s_{r+i}
\end{align*}
Inserting this bound into \eqref{eqn::i2_exp_mom} gives the second term in \eqref{eqn::small_decimation}.
\end{proof}

\begin{lemma}
\label{lem::exp_moment}
There are constants $C_8,C_9,C_{10} >0$ such that for
\[ t= (1+a)C_8 |\CG|(\trel(\CG) +\log|\CG|)\]
and every $x \in V(\CG)$ we have
\begin{equation}
\label{eqn::exp_mom}
\E_x\left[2^{|\CU(t)|}\right]\le 1+ C_9 \exp\left( -a C_{10} \log(n^*) \right).
\end{equation}
\end{lemma}
\begin{proof}
We can write
\[ \E_x\left[ 2^{|\CU(t)|}\right]\le 1+ \sum_{i=1}^{r+\wt r}2^{s_{i-1}}\p\left[|\CU(t)|> s_i\right].  \]

For $i\leq r$, we have that $s_{i-1}= s_i +\tu(\CG)$.  By Lemma~\ref{lem::large_decimation}, we have that
\begin{align*}
2^{s_{i}+\tu(\CG)}\p[|\CU(t)| > t] &\leq \exp\left(  (s_i+\tu(\CG)) \log 2 + \frac{s_i}{\trel(\CG)} \left(C_4 \log|\CG| - \frac{C_5}{|\CG|} t \right) \right).
\end{align*}
By taking $C_8$ (in the statement) large enough, this is in turn bounded from above by
\begin{equation}
\exp \left(-a s_i \left( 1 + \frac{\log |\CG|}{\trel(\CG)} \right) \right). \label{eqn::exp_mom_large_decimation}
\end{equation}

For $r+i\in \{r+1,\dots, r+\wt r\}$ we have from \eqref{eqn::small_decimation} that
\begin{align*}
 &2^{s_{r+i-1}}\p_x[\mathcal |\CU(t)|> s_{r+i} ]  \\
 \leq& 2^{s_{r+i-1}} \p_x[ |\CU(t)| > \tfrac{t}{2}] +
  \exp\left( s_{r+i-1} \left((C_6 + \log 2) i  - \frac{C_7}{|\CG| G^*(n^*)} t \right) \right).
\end{align*}

The first term admits the same bound as \eqref{eqn::exp_mom_large_decimation} with $i=r$, possibly by increasing $C_8$ if necessary.  Using that $i \leq \log_2 |n^*|$, by increasing $C_8$ if necessary, from condition \eqref{assump::main::green_ratio} it is easy to see that the second term admits the bound
\begin{equation}
\label{eqn::exp_mom_small_decimation_2nd_term}
 \exp\left( -a s_{r+i} \frac{\log|\CG| + \trel(\CG)}{G^*(n^*)}  \right).
\end{equation}
Applying condition \eqref{assump::main::green_ratio} again, we see that \eqref{eqn::exp_mom_small_decimation_2nd_term} is bounded from above by
\[ \exp(-a s_{r+i} \log (n^*)).\]

Putting together the estimates we get that for $i \in \{ 1\dots \wt r\}$
\begin{align}
  & 2^{s_{r+i-1}}\p_x[\mathcal |\CU(t)|> s_{r+i} ] \notag\\
 \leq& \exp \left(  - a s_r \left( 1 + \frac{\log |\CG|}{\trel(\CG)} \right) \right) +  \exp\left( -a s_{r+i} \log(n^*) \right) \label{eqn::exp_mom_small_decimation}
\end{align}
Summing \eqref{eqn::exp_mom_large_decimation} and \eqref{eqn::exp_mom_small_decimation} gives \eqref{eqn::exp_mom} (the dominant term in the summation comes from when $s_{r+i} =1$) which proves the lemma.
\end{proof}

\begin{proof}[Proof of Theorem~\ref{thm::main}]
This is a consequence of Lemma~\ref{lem::exp_moment} and the relationship between $\tu(\CG^\diamond)$ and $\E[2^{|\CU(t)|}]$ given in \eqref{eqn::expmom}.
\end{proof}

\bibliographystyle{acmtrans-ims}
\bibliography{lamplighter}

\end{document}